\documentclass[preprint,review,12pt]{elsarticle}

\usepackage{amssymb}
\usepackage[leqno]{amsmath}
\usepackage[german,english]{babel}
\usepackage{graphicx}
\usepackage{subfigure}
\usepackage{cases}
\usepackage{amsfonts}
\usepackage{amsthm}
\usepackage{booktabs}
\usepackage{hyperref}
\usepackage{epstopdf}
\usepackage{color}
\usepackage{float}

\usepackage{mathrsfs}

\numberwithin{equation}{section}

\makeatletter
\def\fnum@figure{Fig.\thefigure}
\makeatother

\theoremstyle{definition}
\newtheorem*{assumption*}{Assumption}

\newtheorem{remark}{Remark}
\numberwithin{remark}{section}
\newtheorem*{remark*}{Remark}

\newtheorem{lemma}{Lemma}
\numberwithin{lemma}{section}
\newtheorem{theorem}{Theorem}
\numberwithin{theorem}{section}

\let \mr=\mathrm
\let \b=\boldsymbol

\begin{document}

\begin{frontmatter}




\title{Supercloseness of the SDFEM on Shishkin triangular meshes for problems with exponential layers}

\author[label1] {Jin Zhang\corref{cor1}}
\author[label2] {Xiaowei Liu\fnref {cor2}}
\cortext[cor1] {Corresponding author:   jinzhangalex@hotmail.com }
\fntext[cor2] {Email: xwliuvivi@hotmail.com }
\address[label1]{School of Mathematical Sciences, Shandong Normal University,
Jinan 250014, China}
\address[label2]{College of Science, Qilu University of Technology, Jinan 250353, China.}

\begin{abstract}
In this paper,  we analyze the supercloseness property of the  streamline diffusion finite element method (SDFEM) on Shishkin triangular meshes, which is different from one in the case of 
rectangular meshes.
 The analysis depends on integral inequalities for the part related to the diffusion in the bilinear form.
Moreover, our result allows the construction of a simple postprocessing that yields a more accurate solution.
Finally, numerical experiments support these theoretical results.

\end{abstract}

\end{frontmatter}
%
%
%

%
%
\section{Introduction}
We consider the singularly perturbed boundary value problem
 \begin{equation}\label{eq:model problem}
 \begin{array}{rcl}
-\varepsilon\Delta u+\boldsymbol{b}\cdot \nabla u+cu=f & \mbox{in}& \Omega=(0,1)^{2},\\
 u=0 & \mbox{on}& \partial\Omega ,
 \end{array}
 \end{equation}
where $\varepsilon\ll |\b{b}|$ is a small positive parameter, the functions $\boldsymbol{b}(x,y)=(b_{1}(x,y),b_{2}(x,y))^{T}$, $c(x,y)\ge 0$ and $f(x,y)$ are supposed sufficiently smooth. We also assume  $$b_{1}(x,y)\ge \beta_{1}>0,\,b_{2}(x,y)\ge \beta_{2}>0, \, c(x,y)-\frac{1}{2}\nabla\cdot \b{b}(x,y)\ge \mu_{0}>0\quad \text{on $\bar{\Omega},$} $$
where $\beta_{1}$, $\beta_{2}$ and $\mu_{0}$ are some constants. The solution of \eqref{eq:model problem} typically has
two exponential layers of width $O(\varepsilon\ln(1/\varepsilon))$ at the sides
$x=1$ and $y=1$ of $\Omega$.  

Because of the presence of layers, standard numerical methods such as the finite element method and the finite difference method, suffer  from severe nonphysciall oscillations.
Thus, stabilized methods and/or a priori adapted meshes (see \cite{Stynes:2005-Steady, Roo1Sty2Tob3:2008-Robust}) are widely used.
In this  paper,  we are to analyze the streamline diffusion finite element method (SDFEM) \cite{Hugh1Broo2:1979-multidimensional} on the Shishkin mesh  \cite{Shishkin:1990-Grid}. This combination possesses good numerical stability and high accuracy for problems \eqref{eq:model problem}, see \cite{Linb1Styn2:2001-Numerical} for detailed numerical tests.

The SDFEM on Shishkin rectangular meshes are widely studied, see \cite{Styn1Tobi2:2003-SDFEM,Fra1Lin2Roo3:2008-Superconvergence,Fra1Kel2Sty3:2012-Galerkin,Styn1Tobi2:2008-Using,Zha1Mei2Che3:2013-Pointwise} and references therein.  
In these papers, supercloseness results  are  analyzed for  optimal $L^2$ estimates,   $L^{\infty}$ bounds and  postprocessing procedures etc. Here ``supercloseness''  means convergence order  of $u^I-u^{N}$ in some norm is greater than one
of $u-u^I$, where $u^I$ is the interpolant of the solution $u$ from the finite element space, $u^{N}$ is the SDFEM solution.  However, to our knowledge,  there are no supercloseness results of the  SDFEM on triangular meshes, which are one kind of popular grids for two-dimensional domains. 
The main reason
is that there are no analysis tools on triangular meshes similar to  Lin's integral identities \cite{Lin1Yan2Zho3:1991-rectangle} which are used to obtain the supercloseness  properties in the case of  rectangles.


In this work,  we present integral inequalities , i.e. Lemma \ref{lem:superconvergence-diffusion-term}, for the diffusion part  in the bilinear form, by means of which the bound $\Vert u^I-u^N\Vert_{SD}\le C  (\varepsilon^{1/2} N^{-1}+N^{-3/2})\ln^{3/2}N$  is obtained. 
Based on this result,  a  simple postprocessing technique  is applied to the SDFEM's solution $u^N$ and this procedure yields  a more accurate numerical solution. Finally, numerical experiments support our theoretical results.

 Here is the outline of this article. In \S 2 we give some a priori informations of the solution of \eqref{eq:model problem}, then introduce Shishkin meshes  and a streamline diffusion finite element method on these meshes. In \S 3 we obtain the supercloseness result. In \S 4 we present the uniform estimate for the postprocessing solution. Finally, some numerical results are presented in \S 5.

Throughout the article, the standard notations for the Sobolev spaces
and norms will be used; and generic constants $C$, $C_i$ are independent of
$\varepsilon$ and $N$.
An index will be attached to indicate an inner product or a norm on a subdomain $D$,
for example, $(\cdot, \cdot)_{D}$ and $\Vert \cdot \Vert_{D}$.

%
%
\section{The SDFEM on Shishkin meshes}
In this section we will introduce the apriori informations of the solution,  the Shishkin mesh and the SDFEM.
\subsection{The regularity results}
As mentioned before the solution $u$ of \eqref{eq:model problem} possesses two
exponential layers at $x=1$ and $y=1$. For our later analysis we shall suppose
that $u$ can be split into a regular solution component and
various layer parts:
\newtheorem{assumption}[theorem]{Assumption}
\begin{assumption}\label{assumption-regularity}
Assume that the solution of \eqref{eq:model problem} can be decomposed as
\begin{equation}\label{eq:(2.1a)}
u=S+E_{1}+E_{2}+E_{12},\quad \forall (x,y)\in\bar{\Omega}.
\end{equation}
For $0\le i+j \le 2$, the regular part satisfies
\begin{equation*}\label{eq:(2.1b)}
\left|\frac{\partial^{i+j}S}{\partial x^{i}\partial y^{j}}(x,y) \right|\le C,
\end{equation*}
while for  $0\le i+j \le 3$,  the layer terms satisfy\upshape{:}
\begin{equation*}\label{eq:(2.1c)}
\left|\frac{\partial^{i+j}E_{1}}{\partial x^{i}\partial y^{j}}(x,y) \right|\le C\varepsilon^{-i}e^{-\beta_{1}(1-x)/\varepsilon},
\end{equation*}
\begin{equation*}\label{eq:(2.1d)}
\left|\frac{\partial^{i+j}E_{2}}{\partial x^{i}\partial y^{j}}(x,y) \right|\le C\varepsilon^{-j}e^{-\beta_{2}(1-y)/\varepsilon},
\end{equation*}
and
\begin{equation*}\label{eq:(2.1e)}
\quad\quad\quad
\left|\frac{\partial^{i+j}E_{12}}{\partial x^{i}\partial y^{j}}(x,y) \right|\le C\varepsilon^{-(i+j)}e^{-(\beta_{1}(1-x)+\beta_{2}(1-y))/\varepsilon}.
\end{equation*}
\end{assumption}
\begin{remark}
Conditions on the data of the problem that guarantee the existence of this decomposition are given in \cite[Theorem 5.1]{Linb1Styn2:2001-Asymptotic}.
\end{remark}

\subsection{Shishkin meshes }
When discretizing \eqref{eq:model problem}, first we divide the domain  $\Omega$
into four subdomains (see Figure \ref{Shishkin mesh})
\begin{align*}
&\Omega_{s}:=\left[0,1-\lambda_{x}\right]\times\left[0,1-\lambda_{y}\right],&&
\Omega_{x}:=\left[ 1-\lambda_{x},1 \right]\times\left[0,1-\lambda_{y}\right],\\
&\Omega_{y}:=\left[0,1-\lambda_{x}\right]\times\left[1-\lambda_{y},1 \right],&&
\Omega_{xy}:=\left[ 1-\lambda_{x},1 \right]\times\left[1-\lambda_{y},1 \right].
\end{align*}
Here $\lambda_x$ and $\lambda_y$ are mesh transition parameters which are used to separate the domain $\Omega$ into the smooth part and different layer parts. They are defined as follows:
\begin{equation*}
\lambda_{x}:=\min\left\{ \frac{1}{2},\rho\frac{\varepsilon}{\beta_{1}}\ln N \right\} \quad \mbox{and} \quad
\lambda_{y}:=\min\left\{
\frac{1}{2},\rho\frac{\varepsilon}{\beta_{2}}\ln N \right\}.
\end{equation*}
In this paper, we set $\rho=2.5$ for technical reasons, see \cite[Remark 2.1]{Zhang:2003-Finite} for the discussions of the different  choices of $\rho$.
\begin{assumption}
Assume that $\varepsilon\le N^{-1}$, as is generally the case in practice. Furthermore assume that
$\lambda_{x}=\rho\varepsilon\beta^{-1}_{1}\ln N$ and $\lambda_{y}=\rho\varepsilon\beta^{-1}_{2}\ln N$ as otherwise $N^{-1}$ is exponentially small compared with $\varepsilon$.
\end{assumption}

Each subdomain is then decomposed into $N/2\times N/2$ ($N\ge 4$ is a positive even integer) uniform rectangles and uniform triangles by drawing the diagonal 
in each rectangles (see Figure \ref{Shishkin mesh}). This yields a piecewise uniform triangulation of $\Omega$  denoted by $\mathcal{T}_{N}$.  Therefore, there are $N^2$ nodes $(x_i,y_j)$, $i,j=0,1,\ldots,N$ and $2N^2$
triangle elements.

We denote $h_{x,i}:=x_{i+1}-x_{i}$ and $h_{y,i}:=y_{i+1}-y_{i}$  which
satisfy
\begin{align*}
&N^{-1}\le h_{x,i},h_{y,i} \le 2N^{-1} \quad \text{if $i=0,1,\ldots,N/2-1$}\\
&C_{1}\varepsilon N^{-1}\ln N \le h_{x,i},h_{y,i}\le C_{2}\varepsilon N^{-1}\ln N.
 \quad \text{if $i=N/2,N/2+1,\ldots,N-1$}
\end{align*}
 For mesh elements we shall use some notations: $K^1_{i,j}$ for the mesh triangle with vertices $(x_i,y_j)$, $(x_{i+1},y_j)$ and $(x_i,y_{j+1})$; $K^2_{i,j}$  for the mesh triangle with vertices
$(x_i,y_{j+1})$, $(x_{i+1},y_j)$ and $(x_{i+1},y_{j+1})$ (see Fig. \ref{fig:code of mesh});  $K$ for a generic mesh triangle.

\begin{figure}
\begin{minipage}[t]{0.5\linewidth}
\centering
\includegraphics[width=2.5in]{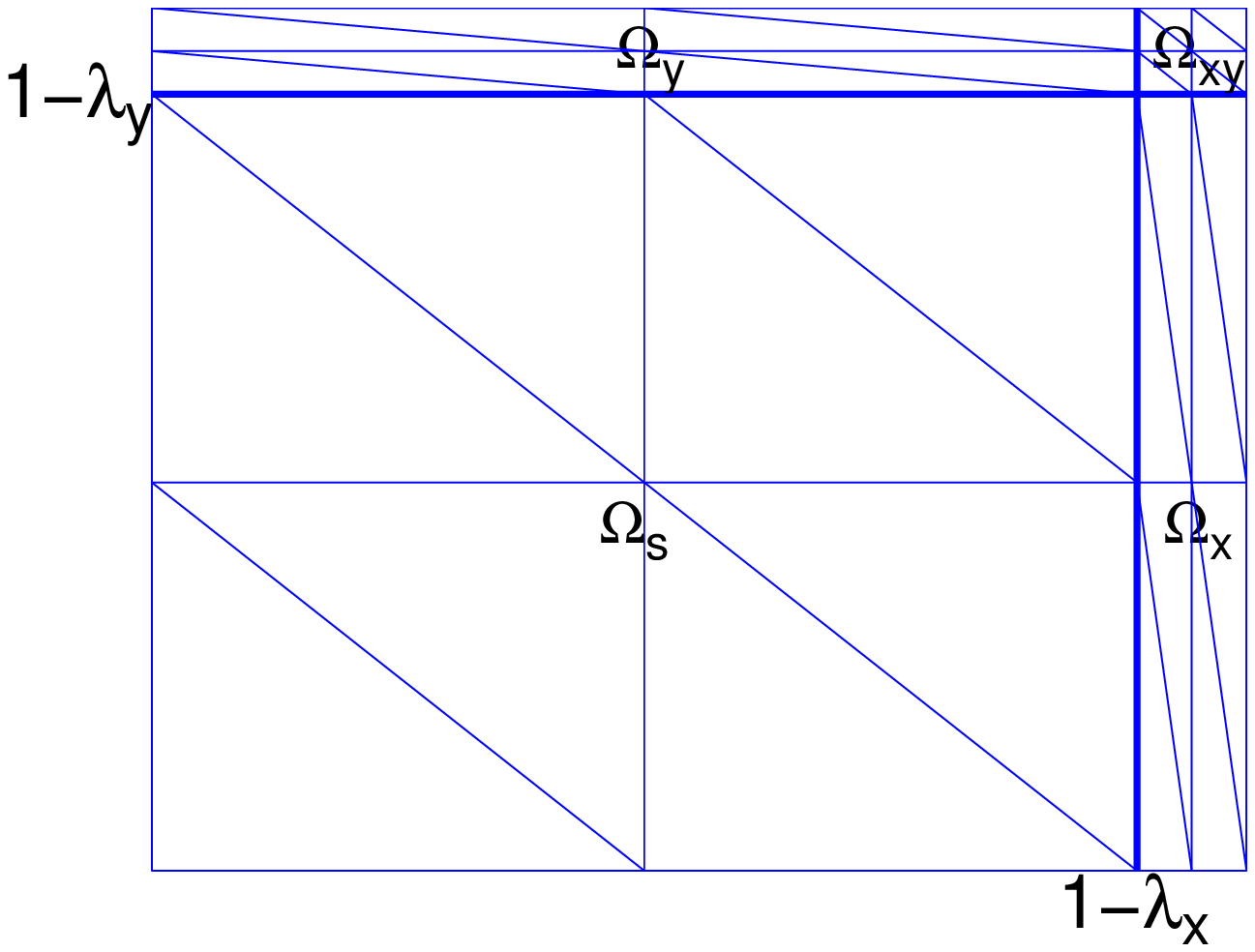}
\caption{Dissection of $\Omega$ and triangulation $\mathcal{T}_{N}$.}
\label{Shishkin mesh}
\end{minipage}%
\begin{minipage}[t]{0.5\linewidth}
\centering
\includegraphics[width=2.5in]{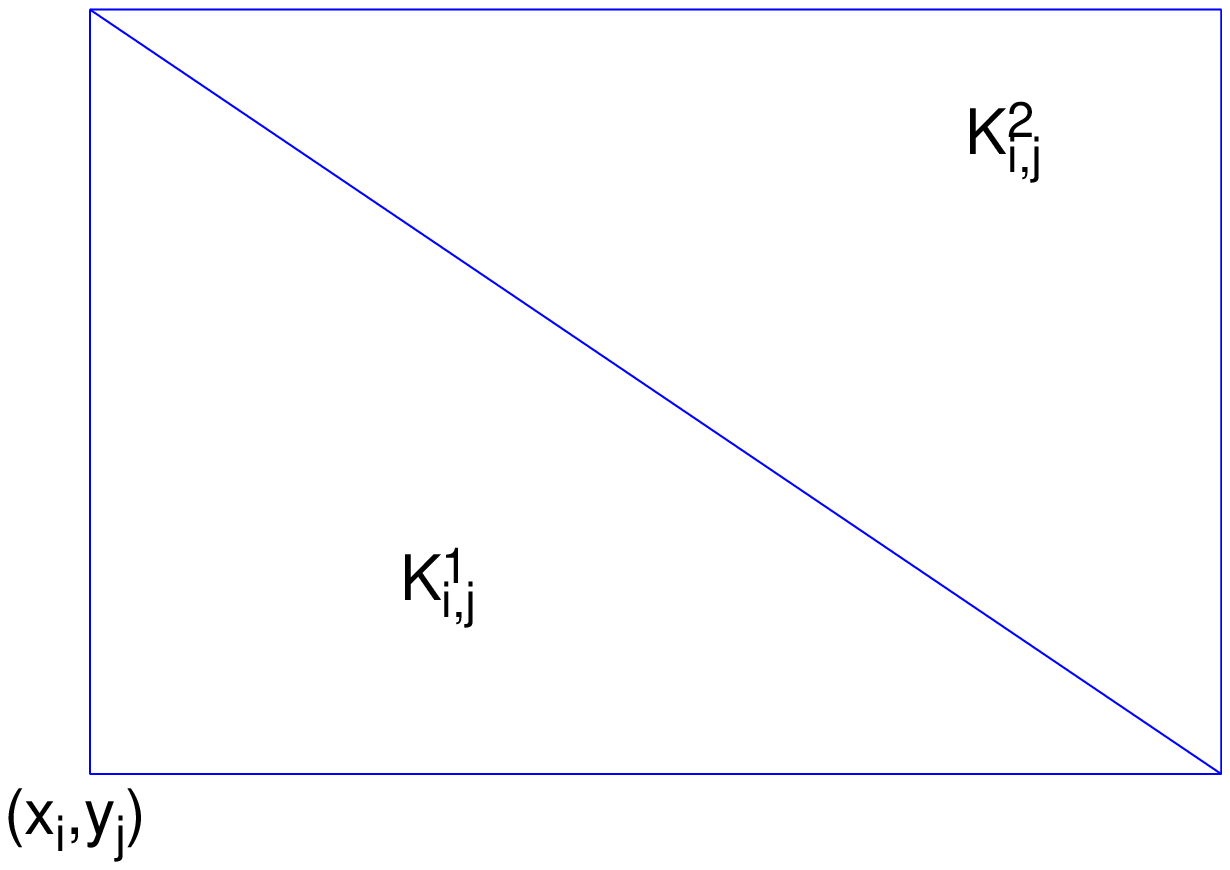}
\caption{$K^{1}_{i,j}$ and $K^{2}_{i,j}$}
\label{fig:code of mesh}
\end{minipage}
\end{figure}

\subsection{The streamline diffusion finite element method}

Let $V:=H^{1}_{0}(\Omega)$. A weak formulation of
the model problem \eqref{eq:model problem} reads:
Find $u\in V$ such that 
\begin{equation}\label{eq:weak formulation}
\varepsilon (\nabla u,\nabla v)+(\boldsymbol{b}\cdot\nabla u+cu,v)=(f,v), \quad \forall  v\in V.
\end{equation}
Note that the variational formulation \eqref{eq:weak formulation} has a unique solution by means of the
Lax-Milgram Lemma.
\par
Let $V^N\subset V$ be the $C^0$ linear finite element space on the Shishkin mesh. The SDFEM reads: Find $u^{N}\in V^{N}$ such that 
\begin{equation}\label{eq:SDFEM}
 a_{SD}(u^{N},v^{N})=(f,v^{N})+\underset{K\subset\Omega}\sum(f,\delta_{K}\boldsymbol{b}\cdot\nabla v^{N})_{K},\quad \forall v^{N}\in V^{N}
\end{equation}
where
\begin{equation*}
a_{SD}(u^{N},v^{N})=a_{Gal}(u^{N},v^{N})+a_{stab}(u^{N},v^{N})
\end{equation*}
and
\begin{align*}
a_{Gal}(u^{N},v^{N})&=\varepsilon (\nabla u^{N},\nabla v^{N})+(\boldsymbol{b}\cdot\nabla u^{N}+cu^{N},v^{N})\\
a_{stab}(u^{N},v^{N})&=\sum_{K\subset\Omega}(-\varepsilon\Delta u^{N}+\boldsymbol{b}\cdot\nabla u^{N}+cu^{N},\delta_{K}\boldsymbol{b}\cdot\nabla v^{N})_{K}.
\end{align*}
Note that $\Delta u^N=0$ in $K$ for $u^N\vert_K\in P_1(K)$.  Following usual practice
\cite{Roo1Sty2Tob3:2008-Robust},  the parameter $\delta_{K}=\delta_{K}(x,y)$ is defined as follows
\begin{equation}\label{eq: delta-K}
\delta_{K}:=
\left\{
\begin{array}{cc}
C^{\ast}N^{-1},&\text{if $K\subset\Omega_{s}$},\\
0,&\text{otherwise},
\end{array}
\right.
\end{equation}
where $C^{\ast}$ is a properly defined positive constant such that the following coercivity holds  (see \cite[Lemma 3.25]{Roo1Sty2Tob3:2008-Robust})
\begin{equation}\label{eq:SD coercivity}
a_{SD}(v^{N},v^{N})\ge \frac{1}{2} \Vert v^{N} \Vert^2_{SD},
\quad \forall v^{N}\in V^{N}.
\end{equation}

We define an energy norm associated with $a_{Gal}(\cdot,\cdot)$ and the streamline diffusion norm (SD norm) associated with $a_{SD}(\cdot,\cdot)$:
\begin{align}
\Vert v^{N} \Vert^{2}_{\varepsilon}:&=
\varepsilon \vert v^{N} \vert^{2}_{1}+ \mu_0\Vert v^{N} \Vert^{2}\label{eq:energy norm},\\
\Vert v^{N} \Vert^{2}_{SD}:&=
\varepsilon \vert v^{N} \vert^{2}_{1}+\mu_{0}\Vert v^{N} \Vert^{2}
+\sum_{K\subset\Omega} \delta_{K}\Vert 
\b{b}\cdot\nabla v^{N}\Vert^{2}_{K}.\label{eq:SD norm}
\end{align}

Form \eqref{eq:weak formulation} and \eqref{eq:SDFEM},  we have the following orthogonality 
\begin{equation}\label{eq:orthogonality of SD}
a_{SD}(u-u^N,v^N)=0,\quad \forall v^N\in V^N.
\end{equation}

\subsection{Preliminary}\label{sub-section: preliminary}
In this subsection, we will present our integral inequalities and some interpolation bounds, which are useful for our main results.  For convenience,  we denote
$$\partial^{l}_x\partial^{m}_y v:=\frac{\partial^{l+m}v}{\partial x^l\partial y^m}.$$

Our later analysis depends on the following integral inequalities, by which
we could obtain sharper estimates for the diffusion part in the bilinear form.
Define $\mathcal{Q}_{i,j}=K^{2}_{i,j-1}\cup K^{1}_{i,j}$ and $\mathcal{S}_{i,j}=K^{2}_{i-1,j}\cup K^{1}_{i,j}$ (see Fig. \ref{fig:Q} and \ref{fig:S}), where $h_{y,j-1}=h_{y,j}$ in $\mathcal{Q}_{i,j}$
and $h_{x,i-1}=h_{x,i}$ in $\mathcal{S}_{i,j}$.

\begin{figure}
\begin{minipage}[t]{0.5\linewidth}
\centering
\includegraphics[width=2.5in]{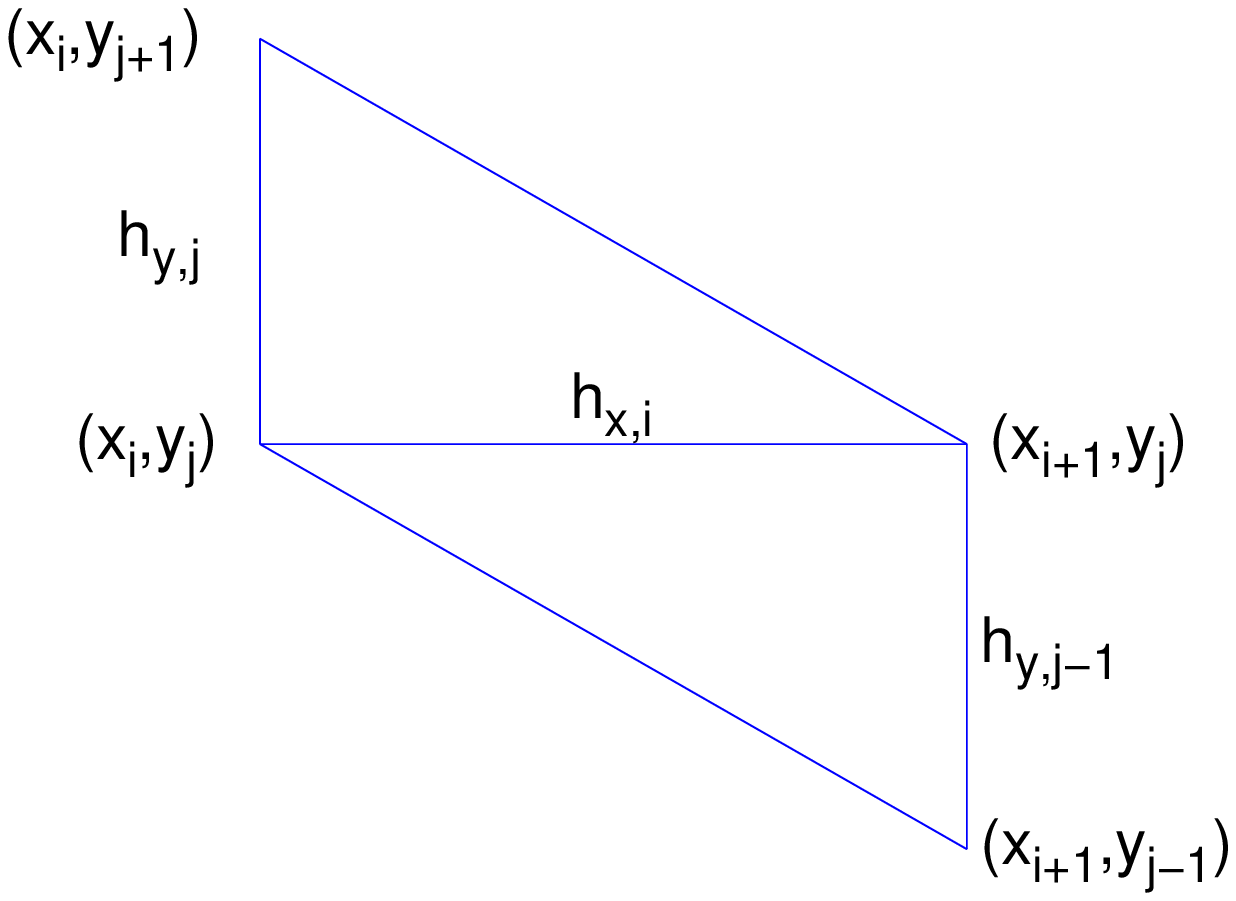}
\caption{Structure of $\mathcal{Q}_{i,j}$}
\label{fig:Q}
\end{minipage}%
\begin{minipage}[t]{0.5\linewidth}
\centering
\includegraphics[width=2.5in]{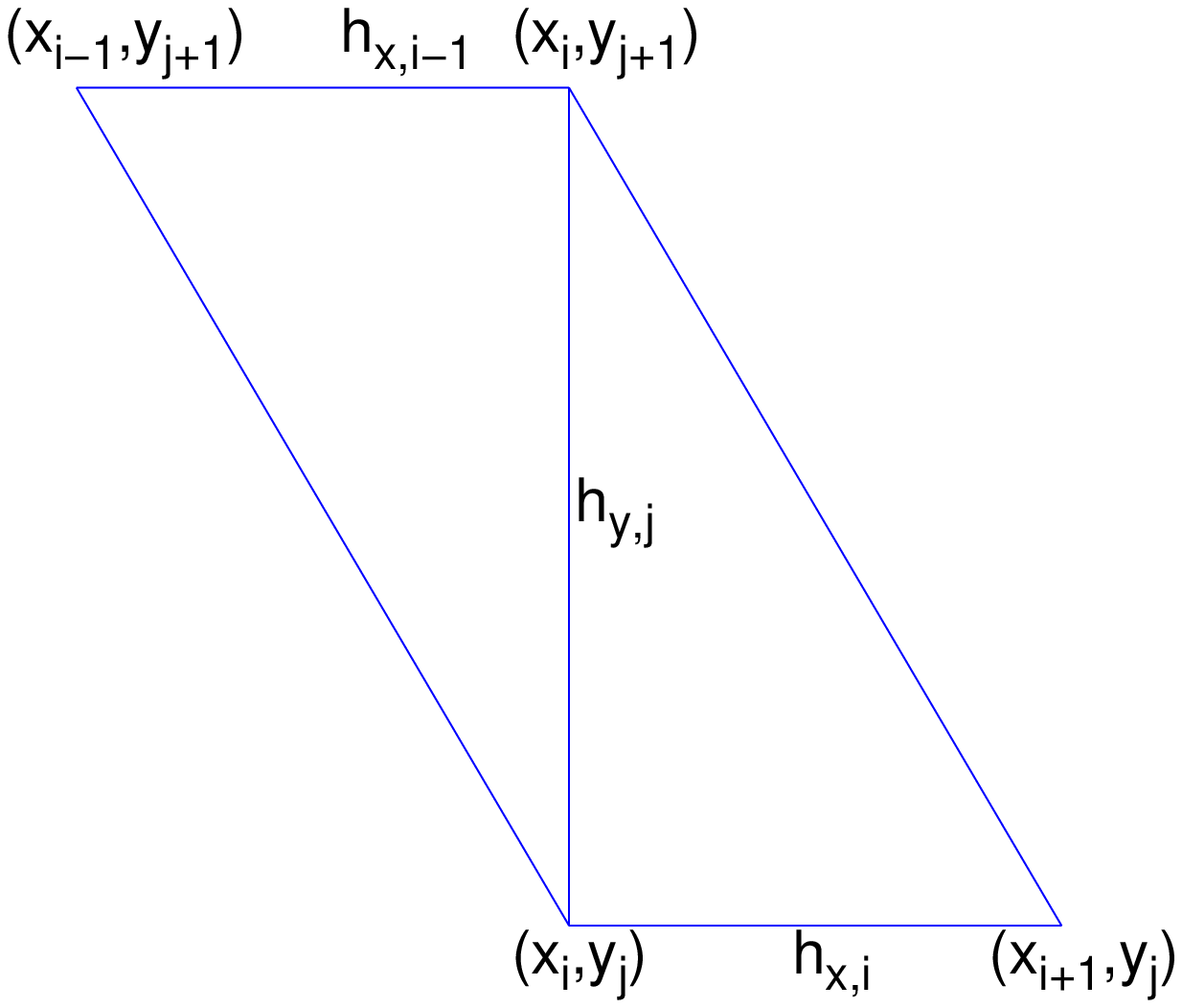}
\caption{Structure of $\mathcal{S}_{i,j}$}
\label{fig:S}
\end{minipage}
\end{figure}

\begin{lemma}\label{lem:superconvergence-diffusion-term}
 Assume that $w\in  C^3(\bar{\Omega})$  and $v^N\in V^N$.   Let $w^I$ be the standard nodal linear interpolation on $\mathcal{T}_N$. Then 
we have
\begin{align}
&\left\vert \int_{\mathcal{Q}_{i,j}}  (w-w^{I})_{x}v^{N}_{x}\mr{d}x\mr{d}y\right\vert
\le 
C\sum_{l+m=2}h^{l}_{x,i}h^{m}_{y,j} \Vert \partial^{l+1}_x\partial^m_y w \Vert_{L^{\infty}(\mathcal{Q}_{i,j})}
\Vert v^{N}_{x} \Vert_{L^{1}(\mathcal{Q}_{i,j})},
\label{eq:diffusion-term x}\\
&\left\vert \int_{\mathcal{S}_{i,j}}  (w-w^{I})_{y}v^{N}_{y}\mr{d}x\mr{d}y\right\vert
\le 
C\sum_{l+m=2}h^{l}_{x,i}h^{m}_{y,j}\Vert \partial^{l}_x\partial^{m+1}_y w \Vert_{L^{\infty}(\mathcal{S}_{i,j})}
\Vert v^{N}_{y} \Vert_{L^{1}(\mathcal{S}_{i,j})}
\label{eq:diffusion-term y}
\end{align}
where $l$ and $m$ are nonnegative integers.
\end{lemma}
\begin{proof}
Recalling $v^N\in V^N$ and noticing that $v^{N}_{x}$ is constant in $\mathcal{Q}_{i,j}$, then we have
\begin{equation}\label{eq: vN-x constant}
\int_{\mathcal{Q}_{i,j}}(w-w^{I})_{x}v^{N}_{x}\mr{d}x\mr{d}y=v^{N}_{x} \int_{\mathcal{Q}_{i,j}}(w-w^{I})_{x}\mr{d}x\mr{d}y. 
\end{equation}
Using Taylor expansion in $(x_{i},y_{j})$ for $(w-w^{I})_{x}$ in $\mathcal{Q}_{i,j}$, we obtain
\begin{align*}
&(w-w^I)_{x}=w_{xx}(x_i,y_j)(x-x_i)+w_{xy}(x_i,y_j)(y-y_j)-\frac{h_{x,i}} {2} w_{xx}(x_i,y_j)+\mathcal{R}_{i,j}
\end{align*}
where 
$$
\Vert \mathcal{R}_{i,j} \Vert_{L^{\infty}(\mathcal{Q}_{i,j})}\le
C\sum_{l+m=2}h^{l}_{x,i}h^{m}_{y,j}\Vert \partial^{l+1}_x\partial^m_y w \Vert_{L^{\infty}(\mathcal{Q}_{i,j})}. 
$$ 
Direct calculations yield
\begin{align}
&\left\vert \int_{\mathcal{Q}_{i,j}}(w-w^I)_{x} \mr{d}x\mr{d}y \right\vert=\left\vert \int_{\mathcal{Q}_{i,j}} \mathcal{R}_{i,j} \mr{d}x\mr{d}y \right\vert
\label{eq:u-uI-x-in Q}\\
\le &
C\mr{meas}(\mathcal{Q}_{i,j})\sum_{l+m=2}h^{l}_{x,i}h^{m}_{y,j} \Vert \partial^{l+1}_x\partial^m_y w \Vert_{L^{\infty}(\mathcal{Q}_{i,j})}
\nonumber.
\end{align}
Combing \eqref{eq: vN-x constant} and \eqref{eq:u-uI-x-in Q},  we obtain \eqref{eq:diffusion-term x}.
The analysis of \eqref{eq:diffusion-term y} is similar to one of \eqref{eq:diffusion-term x}.
\end{proof}
\begin{remark}
Lemma \ref{lem:superconvergence-diffusion-term} 
could be regarded as a simplified version of  \cite[Lemma 2.3]{Bank1Xu2:2003-Asymptotically}  and similar result has appeared in \cite[Lemma 1]{Che1Lin2Zho3etc:2013-Uniform} for uniform meshes.
%
 In \cite{Roos:2006-Superconvergence}, the author combined Lemma 2.3 in \cite{Bank1Xu2:2003-Asymptotically} 
with Bramble-Hilbert Lemma to analyze the diffusion part only in $\Omega_s$.  
Our later analysis, Lemmas \ref{lem: uI-uN-I} and \ref{lem: uI-uN-II},  shows that the most difficult part in the analysis
 is the diffusion part of the bilinear form in $\Omega\setminus \Omega_s$.
\end{remark}

For analysis on Shishkin meshes, we need the following  anisotropic interpolation error bounds given in 
\cite[Lemma 3.2]{Guo1Styn2:1997-Pointwise}.
\begin{lemma}\label{lem: anisotropic interpolation-triangle}
Let $K\in\mathcal{T}_N$ and $p\in (1,\infty]$ and
suppose that $K$ is $K^1_{i,j}$ or $K^2_{i,j}$. Assume
that $w\in W^{2,p}(\Omega)$ and  $w^I$ is the standard nodal linear interpolation on $\mathcal{T}_N$.  Then
\begin{align*}
&\Vert w-w^I \Vert_{L^p(K)}\le
C\sum_{l+m=2}h^l_{x,i}h^m_{y,j}\Vert \partial^l_x\partial^m_y w \Vert_{L^p(K)},\\
&\Vert (w-w^I)_x \Vert_{L^p(K)}\le
C\sum_{l+m=1}h^l_{x,i}h^m_{y,j}\Vert \partial^{l+1}_x\partial^m_y w \Vert_{L^p(K)},\\
&\Vert (w-w^I)_y \Vert_{L^p(K)}\le
C\sum_{l+m=1}h^l_{x,i}h^m_{y,j}\Vert \partial^{l}_x\partial^{m+1}_y w \Vert_{L^p(K)}
\end{align*}
where $l$ and $m$ are nonnegative integers.
\end{lemma}

\begin{lemma}\label{lem: linear interpolation}
Let $u^{I}$ and $E^{I}$ denote the piecewise linear
interpolation of $u$ and $E$, respectively,  on the Shishkin mesh $\mathcal{T}_{N}$, where $E=E_{1}+E_{2}+E_{12}$. Suppose that $u$ satisfies Assumption \ref{assumption-regularity}. Then
\begin{align*}
&\Vert u-u^{I} \Vert_{L^{\infty}(K)}\le
\left\{
\begin{array}{ll}
CN^{-2},&\text{if $K \subset\Omega_{s}$,}\\
CN^{-2}\ln^{2}N, &\text{otherwise;}
\end{array}
\right.\\
&\Vert \nabla E^I \Vert_{L^{1}(\Omega_{s})}\le CN^{-\rho},\quad \Vert (E^{I}_{1})_x \Vert_{L^{\infty}(\Omega_y)}\le C\varepsilon^{-1}N^{-\rho},\\
&\Vert (E^{I}_{2})_x \Vert_{L^{\infty}(\Omega_x)}\le CN^{-\rho},\quad \Vert \nabla E^{I}_{12}  \Vert_{L^{\infty}(\Omega_{x}\cup\Omega_y)}\le C\varepsilon^{-1}N^{-\rho}.
\end{align*}
\end{lemma}
\begin{proof}
The proof of the first inequality is similar to the one of \cite[Theorem 4.2]{Styn1ORior2:1997-uniformly}.
The remained bounds rely on the integral   representation of $\nabla E^I$ and could be obtained directly. 
The reader is  referred to \cite[Lemma 3.2]{Styn1Tobi2:2003-SDFEM} for the basic ideas.
\end{proof}

\section{Supercloseness property}
In this section, we will estimate each term in 
$a_{SD}(u-u^I,v^N)$ to derive the bound of $\Vert u^I-u^N\Vert_{SD}$. First, we
estimate the diffusion part in the bilinear form.
\begin{lemma}\label{lem: uI-uN-I}
Let Assumption \ref{assumption-regularity} hold. We have
$$ \left| \varepsilon (\nabla (u-u^I), \nabla v^{N})\right|
\le
C  (\varepsilon^{1/2} N^{-1}+N^{-3/2})\ln^{3/2}N
\Vert v^{N} \Vert_{SD}.
$$
\end{lemma}
\begin{proof}
We only present the estimates of $\varepsilon ((u-u^I)_{x}, v^{N}_{x})$, for ones of $\varepsilon ((u-u^I)_{y}, v^{N}_{y})$ are similar.

Using the decomposition \eqref{eq:(2.1a)}, for all $v^{N}\in V^{N}$ we have
\begin{equation*}
((u-u^I)_{x}, v^{N}_{x})=\mr{I}+\mr{II}+\mr{III}
\end{equation*}
where
\begin{align*}
\mr{I}=&
\left(  (S-S^I)_{x}, v^{N}_{x}\right)_{\Omega_{s} }+\left(  (S-S^I)_{x}, v^{N}_{x}\right)_{\Omega\setminus\Omega_{s} }
+\left(  (E_{2}-E^I_{2})_{x}, v^{N}_{x}\right)_{\Omega_{y} }\\
&+\left(  (E-E^I)_{x}, v^{N}_{x}\right)_{\Omega_{xy} },
\\
\mr{II}= &\left(  (E-E^I)_{x}, v^{N}_{x}\right)_{\Omega_{s} } 
+\left(  (E_{1}-E^I_{1})_{x}, v^{N}_{x}\right)_{\Omega_{y} }
+\left(  (E_{2}-E^I_{2})_{x}, v^{N}_{x}\right)_{\Omega_{x} }\\
&+\left( (E_{12}-E_{12}^I)_x,v^N_x\right)_{\Omega_x\cup\Omega_y},\\
\mr{III}=&\left(  (E_{1}-E^I_{1})_{x}, v^{N}_{x}\right)_{\Omega_{x} }.
\end{align*}
In the following, we will estimate them term by term.

\noindent
\textbf{Analysis of $\mr{I}$}:\\
The analysis in this part depends on anisotropic interpolation estimates, i.e.,  Lemma \ref{lem: anisotropic interpolation-triangle}.
\begin{align}
|\left(  (S-S^I)_{x}, v^{N}_{x}\right)_{\Omega_{s} }|
\le & C \Vert  (S-S^I)_{x} \Vert_{L^{\infty}(\Omega_{s})} \Vert v^{N}_{x} \Vert_{L^{1}(\Omega_{s})}
\label{eq: diffusion-I-1}\\
\le & C \varepsilon ^{-1/2}N^{-1}\cdot \varepsilon^{1/2} \Vert v^{N}_{x} \Vert_{\Omega_{s}}\nonumber\\
\le & C\varepsilon ^{-1/2}N^{-1} \Vert v^{N} \Vert_{SD}.
\nonumber
\end{align}
\begin{align}
|\left(  (E_{2}-E^I_{2})_{x}, v^{N}_{x}\right)_{\Omega_{y} }|
\le & C \Vert   (E_{2}-E^I_{2})_{x} \Vert_{L^{\infty}(\Omega_{y})} \Vert v^{N}_{x} \Vert_{L^{1}(\Omega_{y})}
\label{eq: diffusion-I-2}\\
\le & CN^{-1}\ln N \cdot (\varepsilon\ln N)^{1/2} \Vert v^{N}_{x} \Vert_{\Omega_{y}}\nonumber\\
\le & CN^{-1}\ln^{3/2} N  \Vert v^{N} \Vert_{SD}.\nonumber
\end{align}
Similarly, we have 
\begin{align}
|\left(  (S-S^I)_{x}, v^{N}_{x}\right)_{\Omega\setminus\Omega_{s} }|
\le& CN^{-1} \ln^{1/2}N\Vert v^{N} \Vert_{SD}\label{eq: diffusion-I-2}\\
|\left(  (E-E^I)_{x}, v^{N}_{x}\right)_{\Omega_{xy} }|
\le
&C\varepsilon^{-1/2}N^{-1}\ln^{3/2}N\Vert v^{N} \Vert_{SD}
\label{eq: diffusion-I-4}.
\end{align}
Combining \eqref{eq: diffusion-I-1}---\eqref{eq: diffusion-I-4}, we obtain
\begin{equation}\label{eq: diffusion-I}
|\mr{I}|\le C\varepsilon^{-1/2} N^{-1} \ln^{3/2}N\Vert v^{N} \Vert_{SD}.
\end{equation}
\noindent
\textbf{Analysis of $\mr{II}$}:\\
The analysis in this part depends on the smallness of layer functions or/and $\mr{meas}(\Omega\setminus\Omega_s)$.
\begin{align}
\left\vert \left(  (E-E^I)_{x}, v^{N}_{x}\right)_{\Omega_{s} } \right\vert
\le& \left( \Vert E_{x}\Vert_{L^{1}(\Omega_{s})} + \Vert E^I_{x}\Vert_{L^{1}(\Omega_{s})}\right) 
\Vert v^{N}_{x} \Vert_{L^{\infty}(\Omega_{s})}
\label{eq: diffusion-II-1}\\
\le &C   N^{-\rho}\cdot  N\Vert v^{N}_{x} \Vert_{\Omega_{s}}\nonumber\\
\le &C  \varepsilon^{-1/2} N^{1-\rho}\Vert v^{N}\Vert_{SD} \nonumber
\end{align}
where we have used inverse estimates \cite[Theorem 3.2.6]{Ciarlet:1978-finite}.
\begin{align}
\left\vert \left(  (E_{1}-E^I_{1})_{x}, v^{N}_{x}\right)_{\Omega_{y} } \right\vert
\le & \left( \Vert (E_{1})_{x}\Vert_{L^{\infty}(\Omega_{y})} + \Vert (E^I_{1})_{x}\Vert_{L^{\infty}(\Omega_{y})}\right) 
\Vert v^{N}_{x} \Vert_{L^{1}(\Omega_{y})}
\label{eq: diffusion-II-2}\\
\le &C  \varepsilon ^{-1}N^{-\rho}\cdot (\varepsilon\ln N)^{1/2}\Vert v^{N}_{x} \Vert_{\Omega_{y}}\nonumber\\
\le & C\varepsilon^{-1}N^{-\rho}\ln^{1/2} N \Vert v^{N} \Vert_{SD}.\nonumber
\end{align}
Similarly,  we have
\begin{align}
&\left\vert \left(  (E_{2}-E^I_{2})_{x}, v^{N}_{x}\right)_{\Omega_{x} } \right\vert
\le CN^{-\rho}\ln^{1/2} N \Vert v^{N} \Vert_{SD},
\label{eq: diffusion-II-3}\\
&\left\vert \left( (E_{12}-E_{12}^I)_x,v^N_x\right)_{\Omega_x\cup\Omega_y} \right\vert
\le 
C\varepsilon^{-1}N^{-\rho}\ln^{1/2} N \Vert v^{N} \Vert_{SD}.
\label{eq: diffusion-II-4}
\end{align}
Estimates \eqref{eq: diffusion-II-1}--\eqref{eq: diffusion-II-4} yield
\begin{equation}\label{eq: diffusion-II}
|\mr{II}|\le C  (\varepsilon^{-1/2} N^{1-\rho}+
\varepsilon^{-1} N^{-\rho}\ln^{1/2}N)
\Vert v^{N}\Vert_{SD}. 
\end{equation}

\noindent
\textbf{Analysis of $\mr{III}$}:\\
 Estimation of $\mr{III}$ depends on Lemma \ref{lem:superconvergence-diffusion-term}. First, we decompose it as follows:
\begin{align*}
& \left( (E_{1}-E^I_{1})_{x},v^{N}_{x} \right)_{\Omega_{x} }
=\sum_{i=N/2}^{N-1}\sum_{j=0}^{N/2-1}\sum_{m=1}^{2}
\left( (E_{1}-E^I_{1})_{x},v^{N}_{x} \right)_{ K_{i,j}^{m} }\\
&=\sum_{i=N/2}^{N-1}
\left( (E_{1}-E^I_{1})_{x},v^{N}_{x} \right)_{ K_{i,0}^{1} }
+\sum_{i=N/2}^{N-1}\left( (E_{1}-E^I_{1})_{x},v^{N}_{x} \right)_{ K_{i,N/2-1}^{2} }\\
&\;\;\;\;+\sum_{i=N/2+1}^{N-1}\sum_{j=1}^{N/2-1}
\left( (E_{1}-E^I_{1})_{x},v^{N}_{x} \right)_{ K_{i,j-1}^{2}\cup K_{i,j}^{1} }\\
&:=\mr{T}_{1}+\mr{T}_{2}+\mr{T}_{3}.
\end{align*}
Recalling $v^{N}|_{\partial\Omega}=0$, we have $v^{N}_{x}|_{K^{1}_{i,0}}=0$, $i=0,\ldots, N-1$ and then
\begin{equation}\label{eq: diffusion-III-T1}
|\mr{T}_{1}|=0.
\end{equation}
Analysis of $\mr{T}_{2}$ is as follows: 
\begin{align}
|\mr{T}_{2}|
&\le   \Vert (E_{1}-E^I_{1})_{x} \Vert_{L^{\infty}(\Omega_{x,r})}
\Vert  v^{N}_{x}  \Vert_{L^{1}(\Omega_{x,r})}
\label{eq: diffusion-III-T2}\\
&\le C\varepsilon ^{-1} N^{-1}\ln N \cdot\left(\varepsilon N^{-1}\ln N\right)^{1/2} \Vert  v^{N}_{x}  \Vert_{\Omega_{x,r}}\nonumber\\
&\le   C \varepsilon ^{-1} N^{-3/2}\ln ^{3/2} N\Vert v^{N} \Vert_{SD}\nonumber
\end{align}
where $\Omega_{x,r}=\bigcup_{i=N/2}^{N-1}K_{i,N/2-1}^{2}$,
$\mr{meas}(\Omega_{x,r})\le C\varepsilon N^{-1}\ln N$ and we have used  Lemma \ref{lem: anisotropic interpolation-triangle} for $\Vert (E_{1}-E^I_{1})_{x} \Vert_{L^{\infty}(\Omega_{x,r})}$.

Using Lemma \ref{lem:superconvergence-diffusion-term}, we obtain
\begin{align}
|\mr{T}_{3}|\le &C \sum_{i=N/2+1}^{N-1}\sum_{j=1}^{N/2-1} \varepsilon^{-1} N^{-2}\ln^{2}N 
\Vert v^{N}_{x} \Vert_{L^{1}(K_{i,j-1}^{2}\cup K_{i,j}^{1})}
\label{eq: diffusion-III-T3}\\
&\le  C\varepsilon^{-1} N^{-2}\ln^{2}N \Vert v^{N}_{x} \Vert_{L^{1}(\Omega_{x})}\nonumber\\
&\le C\varepsilon^{-1} N^{-2}\ln^{2}N (\varepsilon\ln N)^{1/2}\Vert v^{N}_{x} \Vert_{\Omega_{x}}\nonumber\\
&\le C\varepsilon^{-1} N^{-2}\ln^{5/2}N  \Vert v^{N} \Vert_{SD}.\nonumber
\end{align}
Combining \eqref{eq: diffusion-III-T1}---\eqref{eq: diffusion-III-T3}, we obtain
\begin{equation}\label{eq: diffusion-III}
|\mr{III}|\le C \varepsilon ^{-1} N^{-3/2}\ln ^{3/2} N\Vert v^{N} \Vert_{SD}.
\end{equation}
Collecting \eqref{eq: diffusion-I}, \eqref{eq: diffusion-II} and \eqref{eq: diffusion-III},  we are done.
\end{proof}
\begin{remark}
If we make use of  standard arguments, then we have
\begin{align*}
|\left( (E_{1}-E^I_{1})_{x},v^{N}_{x} \right)_{\Omega_{x} }|
&\le 
\Vert (E_{1}-E^I_{1})_{x} \Vert_{L^{\infty}(\Omega_x)}
\Vert v^{N}_{x} \Vert_{L^{1}(\Omega_x)}\\
&\le
C\varepsilon^{-1}N^{-1}\ln N\cdot (\varepsilon\ln N)^{1/2} 
\Vert v^{N}_{x} \Vert_{ \Omega_x }\\
&\le
C\varepsilon^{-1}N^{-1}\ln^{3/2}N \Vert v^N \Vert_{SD},
\end{align*}
where we have used Lemma \ref{lem: anisotropic interpolation-triangle}.  Thus, we only obtain
$$
\varepsilon|(\nabla(u-u^I), \nabla v^N)|\le CN^{-1}\ln^{3/2}N \Vert v^N \Vert_{SD}.
$$

\end{remark}

Next, we analyze the remained in the bilinear form $a_{SD}(u-u^I,v^N)$.
\begin{lemma}\label{lem: uI-uN-II}
Let Assumption \ref{assumption-regularity} hold true. We have
\begin{align}
&\left| 
(\b{b}\cdot \nabla(u-u^{ I }), v^{N} )
+(c(u-u^{I }), v^{N})\right|
\le CN^{-3/2}
\Vert v^{N} \Vert_{SD},
\label{eq:CR term in uI-uN}\\
&\left| 
a_{stab}(u-u^{I }, v^{N})
\right|
\le C N^{-1/2}(\varepsilon+N^{-1})\Vert v^{N} \Vert_{SD}.
\label{eq:stabilization term in uI-uN}
\end{align}
\end{lemma}

\begin{proof}
%
Integrations by part  and Lemma \ref{lem: linear interpolation} give
\begin{align*}
&\left| (\b{b}\cdot \nabla(u-u^{ I }), v^{N} )
+(c(u-u^{I }), v^{N}) \right|\\
\le&
| (u-u^{I},\b{b}\cdot  \nabla v^{N} )_{\Omega_{s}} |
+
| (u-u^{I},\b{b}\cdot  \nabla v^{N} )_{\Omega\setminus\Omega_{s}} |
+|((c-\nabla\cdot\b{b})(u-u^{I }), v^{N})| \\
\le &
C\Vert  u-u^{I} \Vert_{\Omega_{s}} \Vert  \b{b}\cdot  \nabla v^{N} \Vert_{\Omega_{s}}
+C\Vert  u-u^{I} \Vert_{L^{\infty}(\Omega\setminus \Omega_{s}) } \Vert  \b{b}\cdot  \nabla v^{N} \Vert_{L^{1}(\Omega\setminus \Omega_{s})}+C \Vert  u-u^{I } \Vert \Vert v^{N} \Vert \\
\le&
C(N^{-3/2}+ N^{-2}\ln^{5/2}N+N^{-2}\ln^{2} N)
\Vert v^{N} \Vert_{SD}.
\end{align*}
Thus, \eqref{eq:CR term in uI-uN} is obtained.

Analysis of \eqref{eq:stabilization term in uI-uN} is direct and we  can refer to \cite[Lemma 4.4]{Styn1Tobi2:2003-SDFEM} for detailed analysis.
Recalling $\delta_{K}=0$ if $K\subset\Omega\setminus \Omega_{s}$ and the decomposition \eqref{eq:(2.1a)}, we have
\begin{align}
|a_{stab}(S-S^{I},v^{N})|
&\le 
CN^{-1}\left( \varepsilon+N^{-1}+N^{-2}\right) 
\Vert \b{b}\cdot \nabla v^{N} \Vert_{L^{1}(\Omega_{s})}
\label{eq:a-stab S}\\
&\le
C N^{-1/2}(\varepsilon+N^{-1})\Vert v^{N} \Vert_{SD}
\nonumber
\end{align}
and
\begin{align}
|a_{stab}(E-E^{I},v^{N})|
\le 
CN^{-1} N^{-\rho} \Vert \b{b}\cdot \nabla v^{N} \Vert_{L^{\infty}(\Omega_{s})}
\le
C N^{1/2-\rho}\Vert v^{N} \Vert_{SD}
\label{eq:a-stab E}
\end{align}
where $E=E_1+E_2+E_{12}$. 
Combing \eqref{eq:a-stab S} and \eqref{eq:a-stab E}, we obtain \eqref{eq:stabilization term in uI-uN}.
\end{proof}

Now, we are in a position to state our main result.
\begin{theorem}\label{eq: uI-uN varepsilon}
Let Assumption \ref{assumption-regularity} hold true. We have
$$
\Vert u^I-u^{N} \Vert_{\varepsilon}
\le
\Vert u^I-u^{N} \Vert_{SD}
\le
C  (\varepsilon^{1/2} N^{-1}+N^{-3/2})\ln^{3/2}N. $$
\end{theorem}
\begin{proof}
Considering \eqref{eq:energy norm}, \eqref{eq:SD norm} and coercivity \eqref{eq:SD coercivity} and orthogonality \eqref{eq:orthogonality of SD} of $a_{SD}(\cdot, \cdot)$, we have
\begin{align*}
\frac{1}{2}\Vert  u^{I}-u^{N} \Vert^{2}_{SD}
&\le
a_{SD}(u^{I}-u^N, u^{I}-u^{N})\\
&=a_{SD}(u^{I}-u, u^{I}-u^{N})+a_{SD}(u-u^N, u^{I}-u^{N})\\
&=a_{SD}(u^{I}-u, u^{I}-u^{N})
\end{align*}

Taking $v^N=u^I-u^N$ in Lemma \ref{lem: uI-uN-I} and \ref{lem: uI-uN-II},  we
have
$$
\frac{1}{2}\Vert  u^{I}-u^{N} \Vert^{2}_{SD}
\le 
C  (\varepsilon^{1/2} N^{-1}+N^{-3/2})\ln^{3/2}N\Vert  u^{I}-u^{N} \Vert_{SD}.
$$
Considering $\Vert  u^{I}-u^{N} \Vert_{\varepsilon} 
\le  \Vert  u^{I}-u^{N} \Vert_{SD} $, we are done.
\end{proof}
\begin{remark}\label{remark:triangle and rectangle}
 The  convergence order of $\Vert u^I-u^N \Vert_{\varepsilon}$ is only $3/2$, which also could be observed in our numerical tests (see Table \ref{table: alex-1}) and is different from $2$ order convergence in the case of SDFEM on Shishkin rectangular meshes (see \cite[Theorem 4.5]{Styn1Tobi2:2003-SDFEM}). 
\end{remark}
\begin{remark}
Using Lemma \ref{lem: anisotropic interpolation-triangle}, we have
$$
\Vert u-u^I \Vert_{\varepsilon}\le C N^{-1}\ln N.
$$
Thus we can see the supercloseness property clearly.  Considering Theorem
\ref{eq: uI-uN varepsilon},  we obtain
$$
\Vert u-u^N \Vert_{\varepsilon}\le \Vert u-u^I \Vert_{\varepsilon}+\Vert u^I-u^N \Vert_{\varepsilon}\le C N^{-1}\ln N.
$$

\end{remark}

%

%
%
\section{Errors of postprocessing solution}\label{sec: postprocess}
In this section, we will analyze the uniform estimate of $\Vert  u-\tilde{u}^{N} \Vert_{\varepsilon}$ where $\tilde{u}^{N}$ is the new numerical
solution obtained by applying to $u^{N}$ a local postprocessing technique. 
The procedure of postprocessing is similar to one in \cite[Section 5.2]{Styn1Tobi2:2003-SDFEM}.

Consider a family of Shishkin meshes $\mathcal{T}_{N}$ with mesh points $(x_{i},y_{j})$ for $i, j =
0, \ldots, N$, where we require $N/2$ to be even. Then we can build a coarser mesh
composed of disjoint macrotriangles $M$, each comprising four mesh triangles from
$\mathcal{T}_{N}$, where $M$ belongs to only one of the four domains $\Omega_{s}$, $\Omega_{x}$, $\Omega_{y}$, and $\Omega_{xy}$. Associate with each macrotriangle $M$ an interpolation operator $\mathcal{P}_{M}: C(M)\rightarrow P_{2}(M)$ defined by the standard quadratic interpolation at the nodes, and midpoints of edges of the macrotriangle, where $P_{2}(M)$ consists of  polynomials of degree 2 in two variables. As usual, $\mathcal{P}_{M}$ can be extended
to a continuous global interpolation operator $\mathcal{P}: C(\bar{\Omega})\rightarrow W^{N}$, where $W^{N}$ is the space of piecewise quadratic finite elements, by setting
\begin{equation*}
(\mathcal{P}v)\vert_{M}:=\mathcal{P}_{M}(v\vert_{M}),\quad \forall M.
\end{equation*}
For convenience we define $\tilde{v}:=Pv$. Note that the macrotriangle $M$ does not belong to $\mathcal{T}_{N/2}$
 because the transition point values $1-\lambda_{x}$ and $1-\lambda_{y}$ associated with the Shishkin mesh $\mathcal{T}_{N}$ change when $N$ is replaced by $N/2$.  We shall use the notation $\tilde{\mathcal{T}}_{N/2}$ (see Fig. \ref{fig:macroelement})
for the family of macromeshes
that is generated by the family of Shishkin meshes
$\mathcal{T}_{N}$. Thus each macrotriangle
$M\in \tilde{\mathcal{T}}_{N/2}$ is the union of four triangles from $\mathcal{T}_{N}$.
\begin{figure}[h]
\centering
\includegraphics[width=0.6\textwidth]{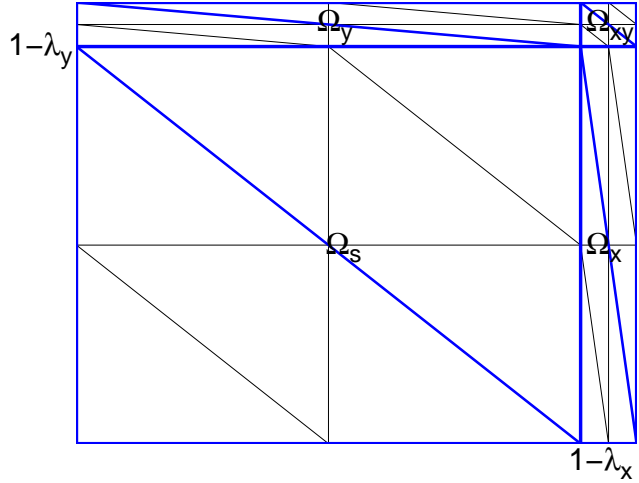}
\caption{Triangulation $\tilde{\mathcal{T}}_{N/2}$.}
\label{fig:macroelement}
\end{figure}

\begin{lemma}\label{lem:properties of P}
The interpolation operator $\mathcal{P}$ has the following properties:
\begin{align*}
&\mathcal{P}(v^I)=P(v),&&\forall v\in C(\bar{\Omega}),\\
&\Vert \mathcal{P} v^{N} \Vert_{\varepsilon}\le  C\Vert  v^{N} \Vert_{\varepsilon},\quad &&\forall v^{N}\in V^{N},\\
&\Vert \mathcal{P} v^{N} \Vert_{SD}\le  C\Vert  v^{N} \Vert_{SD},\quad &&\forall v^{N}\in V^{N}.
\end{align*}
\end{lemma}
\begin{proof}
The proof is standard and the reader is referred to \cite[Lemma 5.5]{Styn1Tobi2:2003-SDFEM}. We just need to consider the differences between standard basis functions on 
triangular meshes and rectangular ones.
\end{proof}

\begin{lemma}\label{lem:u-Pu in varepsilon norm}
Let Assumption \ref{assumption-regularity}  hold true for $0\le i+j\le 3$. Then
\begin{equation*}
\Vert u-\tilde{u}^I \Vert_{\varepsilon}\le C (\varepsilon N^{-3/2}+N^{-2}\ln^2 N).
\end{equation*}
\end{lemma}
\begin{proof}
The proof is similar to the one of \cite[Lemma 5.5]{Styn1Tobi2:2003-SDFEM}.
\end{proof}

\begin{theorem}
Let Assumption \ref{assumption-regularity}  hold true for $0\le i+j\le 3$. Then the
the numerical solution $\tilde{u}^{N}=Pu^{N}$, which is generated by postprocessing the SDFEM's solution $u^{N}$,
satisfies
$$ 
\Vert u-\tilde{u}^{N}\Vert_{\varepsilon} \le C(\varepsilon^{1/2} N^{-1}+N^{-3/2})\ln^{3/2}N.$$
\end{theorem}
\begin{proof}
The triangle inequality, Lemmas  \ref{eq: uI-uN varepsilon}, \ref{lem:properties of P} and \ref{lem:u-Pu in varepsilon norm} yield
\begin{align*}
\Vert u-Pu^{N}\Vert_{\varepsilon}
\le&
\Vert u-Pu\Vert_{\varepsilon}+\Vert P(u^I-u^{N} )\Vert_{\varepsilon}
\le 
\Vert u-Pu\Vert_{\varepsilon}+C\Vert u^I-u^{N}\Vert_{\varepsilon}
\\
\le&
C(\varepsilon N^{-3/2}+N^{-2}\ln^2 N)+C(\varepsilon^{1/2} N^{-1}+N^{-3/2})\ln^{3/2}N\\
\le&
C(\varepsilon^{1/2} N^{-1}+N^{-3/2})\ln^{3/2}N.
\end{align*}
\end{proof}

%
%

%
%
\section{Numerical results}
\noindent
In this section we give numerical results that appear to support our theoretical results. Errors and convergence rates of $u^I-u^N$, $u-u^N$ and $u-Pu^N$ are presented. For the computations we have chosen $C^{\ast}=1.0$ in \eqref{eq: delta-K}. All calculations were carried out using Intel visual Fortran 11. The discrete problems
were solved by the nonsymmetric iterative solver GMRES(cf. e.g.,\cite{Ben1Gol2Lie3:2005-Numerical,Saad1Schu2:1986-GMRES}).
\par
We will  illustrate our results  by computing errors and convergence orders for
the following boundary value problems
\begin{align*}
-\varepsilon\Delta u+2u_{x}+u_{y}+u&=f(x,y)\quad&&\text{in $\Omega=(0,1)^{2}$},\\
u&=0&& \text{on $\partial\Omega$}
\end{align*}
where the right-hand side  $f$ is chosen such that
\begin{equation*}
u(x,y)=2\sin x\left(1-e^{-\frac{2(1-x)}{\varepsilon} } \right)
y^{2} \left( 1-e^{-\frac{(1-y)}{\varepsilon} } \right)
\end{equation*}
is the exact solution.

%

\begin{table}
\caption{$\varepsilon=10^{-4}, 10^{-6},10^{-8},10^{-10}$}
\footnotesize
\begin{tabular*}{\textwidth}{@{\extracolsep{\fill}}ccccc}
\hline
$N$   & $ \Vert u^I-u^{N}\Vert_{\varepsilon}$   & Rate   & $ \Vert u^I-u^{N}\Vert_{SD}$    & Rate\\
\hline
8   &$1.0496\times10^{-1}$  & $0.74$    &$1.2058\times10^{-1}$      &    $0.93$\\
16  &$6.2921\times10^{-2}$   &$1.19$  &$6.3435\times10^{-2}$      &    $1.13$\\
32  &$2.8978\times10^{-2}$   &$1.30$   &$2.9027\times10^{-2}$     &    $1.30$\\
64  &$1.1762\times10^{-2}$   &$1.38$&$1.1769\times10^{-2}$       &    $1.38$\\
128  &$4.5131\times10^{-3}$   &$1.41$&$4.5143\times10^{-3}$      &    $1.41$\\
256  &$1.6965\times10^{-3}$   &$1.42$&$1.6967\times10^{-3}$            &$1.42$\\
512  &$6.3617\times10^{-4}$   &$1.41$&$6.3620\times10^{-4}$            &$1.41$\\
1024  &$2.3980\times10^{-4}$   &$---$&$2.3981\times10^{-4}$            &$---$\\
\hline
\end{tabular*}
\label{table: alex-1}
\end{table}
In Table \ref{table: alex-1}, the errors and convergence rates for $\Vert u^I-u^{N}\Vert_{\varepsilon}$ and $\Vert u^I-u^{N} \Vert_{SD}$ are displayed.  We observe $\varepsilon$-independence of the errors and convergence rates. These numerical results support our theoretical ones: almost $3/2$ order
 convergence for $\Vert u^I-u^{N}\Vert_{\varepsilon}$ and $\Vert u^I-u^{N} \Vert_{SD}$.

\begin{table}
\caption{$\varepsilon=10^{-4}, 10^{-6},10^{-8},10^{-10}$}
\footnotesize
\begin{tabular*}{\textwidth}{@{\extracolsep{\fill}}ccccc}
\hline
$N$   & $ \Vert u-u^{N}\Vert_{\varepsilon}$   & Rate   & $ \Vert u-\tilde{u}^{N}\Vert_{\varepsilon}$    & Rate\\
\hline
8   &$3.05\times10^{-1}$  & $0.53$    &$1.55\times10^{-1}$      &    $0.79$\\
16  &$2.11\times10^{-1}$   &$0.63$  &$8.95\times10^{-2}$      &    $1.09$\\
32  &$1.36\times10^{-1}$   &$0.70$   &$4.19\times10^{-2}$     &    $1.33$\\
64  &$8.38\times10^{-2}$   &$0.75$&$1.67\times10^{-2}$       &    $1.45$\\
128  &$4.99\times10^{-2}$   &$0.78$&$6.12\times10^{-3}$      &    $1.51$\\
256  &$2.90\times10^{-2}$   &$0.81$&$2.15\times10^{-3}$            &$1.53$\\
512  &$1.65\times10^{-2}$   &$0.83$&$7.46\times10^{-4}$            &$1.52$\\
1024  &$9.28\times10^{-3}$   &$---$&$2.60\times10^{-4}$            &$---$\\
\hline
\end{tabular*}
\label{table: alex-2}
\end{table}

Table \ref{table: alex-2} gives the errors and convergence rates for $\Vert u-u^{N}\Vert_{\varepsilon}$ and $\Vert u-\tilde{u}^{N}\Vert_{\varepsilon}$.  We can see that
the convergence  order of $\Vert u-u^{N}\Vert_{\varepsilon}$ is almost $1$ and one
of $\Vert u-\tilde{u}^{N}\Vert_{\varepsilon}$ is almost $3/2$, as supports our theoretical results about the postprocessing solution $\tilde{u}^N$.

\end{document}